\documentclass[12pt,a4paper]{scrartcl}
\usepackage[T2A]{fontenc}
\usepackage[utf8]{inputenc}
\usepackage[english,russian]{babel}
\usepackage{indentfirst}
\usepackage{misccorr}
\usepackage{graphicx}
\usepackage{amsmath}
\usepackage{mathrsfs}
\usepackage[unicode, pdftex]{hyperref}
\usepackage[dvipsnames]{xcolor}
\usepackage{amsthm}

\hypersetup{linktocpage}
\hypersetup{colorlinks=true, linkcolor=blue, citecolor=blue, urlcolor=blue}

\numberwithin{equation}{section}
\newtheorem{theorem}{Теорема}[section]

\newtheorem{proposition}{Предложение}[section]

\newtheorem{remark}{Замечание}[section]
\newtheorem{corollary}{Следствие}[section]

\newcommand{\partialder}[2]{\frac{\partial #1}{\partial #2}}

\newcommand{\partialdersing}[1]{\frac{\partial}{\partial #1}}

\DeclareMathOperator{\R}{\mathbb{R}}
\DeclareMathOperator{\const}{const}

\DeclareMathOperator{\eps}{\varepsilon}
\DeclareMathOperator{\algLie}{\mathfrak{g}}
\DeclareMathOperator{\set}{\Omega}
\DeclareMathOperator{\polarSet}{\Omega^{\circ}}

\DeclareMathOperator{\arcosh}{arcosh}
\DeclareMathOperator{\argmax}{argmax}
\DeclareMathOperator{\Length}{Len}
\DeclareMathOperator{\LenCurve}{L}
\DeclareMathOperator{\Area}{Area}
\DeclareMathOperator{\Lim}{Lim}

\begin{document}

\title{Об одной изопериметрической задаче на плоскости Лобачевского с левоинвариантной финслеровой структурой
\thanks{Исследование выполнено за счет гранта Российского научного фонда (проект № 20-11-20169).}	
}
\date{}
\author{В.~А.~Мырикова \footnote{Математический институт им. В.А. Стеклова Российской академии наук. E-mail: myrikova.va@gmail.com}}
\maketitle
\tableofcontents
\newpage

\section{Введение}
Данная работа продолжает изучение геометрии на финслеровом аналоге плоскости Лобачевского. Некоторые свойства, в том числе и финслеровы геодезические, были изучены ранее, например, в \cite{Gribanova} и \cite{lobachevsky-geodesics}. Основная идея, успешно примененная в указанных работах, -- представить плоскость Лобачевского как группу Ли с левоинвариантной финслеровой структурой. 

А именно, рассматривается группа Ли $G$ собственных афинных преобразований прямой: $\lambda\in\mathbb{R}\mapsto y\lambda +x\in\mathbb{R}$, где элемент $(x,y)$ лежит в верхней полуплоскости. Любая норма, заданная в алгебре Ли группы $G$, порождает левоинвариантную метрику с помощью левых сдвигов. Несложно проверить, что классическая евклидова норма приведет таким построением к стандартной римановой метрике на плоскости Лобачевского. Если же в алгебре Ли задана произвольная неевклидова норма (или даже почти норма), мы получим некоторое финслерово пространство, которое естественно считать финслеровым аналогом плоскости Лобачевского.

В данной работе предпринимается попытка поставить и изучить изопериметрическую задачу на $G$ с описанной выше финслеровой структурой. Конечно, основная трудность заключается в том, чтобы понять, какое определение площади на финслеровом многообразии было бы логично использовать. Имея групповую структуру и левоинвариантную метрику, нам кажется естественным поставить левоинвариантную изопериметрическую задачу, то есть использовать левоинвариантную форму площади. 

Таким образом будет получена геометрическая постановка \eqref{geom-problem-statement}, которую мы далее переформулируем как задачу быстродействия \eqref{optimal-control-problem-statement}. Применение принципа максимума Понтрягина и аппарата выпуклой тригонометрии, предложенного в \cite{convexTrig}, позволят нам получить изопериметрические контуры в явном виде (теорема \ref{theorem-optimal-curves}). Обобщение изопериметрического неравенства будет получено в параметрической форме (теорема \ref{theorem-iso-ineq}).

Важно отметить работу \cite{Busemann1947TheIP}, где в частности были найдены изопериметрические контуры и изопериметрическое неравенство на плоскости Минковского с евклидовой площадью. Мы сравним наши результаты с результатами указанной работы.

\section{Благодарности}
Автор сердечно благодарит своего научного руководителя, Льва Вячеславовича Локуциевского, за постановку задачи, неугасаемое внимание к работе и поддержку в минуты отчаяния. 

\section{Геометрическая постановка задачи}
Рассмотрим группу Ли $G$ собственных афинных преобразований прямой $\lambda\mapsto y\lambda +x$, то есть многообразие $\{(x,y),~x\in\R,~y>0\}$ с групповой операцией $(x_1,y_1)\cdot(x_2,y_2)=(x_1+x_2y_1,y_1y_2)$. Обозначим единицу группы $e=(0,1)$ и алгебру Ли $\algLie=T_eG\cong\R^2$.

Определим на $G$ левоинвариантную финслерову структуру следующим образом. Пусть задано $\Omega\subset\algLie$ - произвольное выпуклое компактное множество, содержащее ноль во внутренности. Тогда определим $\Omega_g\subset T_gG,g\in G$ из условия левоинвариантности: $\Omega_g=d_eL_g(\Omega)$, где $L_g$ -- левый сдвиг на элемент $g$. Полученное семейство множеств задает в касательном пространстве в каждой точке $g\in G$ почти норму $\|\cdot\|_g$ (в точности норму, если $\Omega=-\Omega$) как функцию Минковского множества $\Omega_g$. В координатах получаем $\|(\xi_1,\xi_2)\|_{(x,y)}=\frac{1}{y}\|(\xi_1,\xi_2)\|_e$.

Во избежание путаницы с перегруженным термином <<кривая>>, явно укажем используемые наименования. Кривая -- непрерывное отображение промежутка в $\R^n$. Контур -- образ замкнутой кривой, которую будем называть параметризацией этого контура. Простой контур -- образ замкнутой кривой без самопересечений, кроме начала и конца. Липшицевый контур -- образ липшицевой замкнутой кривой.

Изопериметрическая задача -- это задача поиска простого контура наименьшей длины при фиксированной площади ограничиваемой им области. Контуры с данным свойством будем называть изопериметрическими контурами. Таким образом, для математической постановки задачи необходимо уметь измерять длины контуров и площади областей.

Пусть $\gamma(t)=(x(t),y(t))\in G$, $t\in [a,b]$ -- произвольная липшицева кривая. Финслерова структура позволяет естественно определить ее длину
$$
\LenCurve(\gamma(t))=\int\limits_a^b\|\dot{\gamma}(t)\|_{\gamma(t)}dt=\int\limits_a^b\frac{\|(\dot{x}(t),\dot{y}(t))\|_e}{y(t)}dt.
$$
Важно помнить, что в финслеровом случае длины кривых с противоположными ориентациями, вообще говоря, различны. А потому для каждого липшицева контура $\gamma$ определены две длины $\Length_+(\gamma)=\LenCurve(\gamma_+(t))$ и $\Length_-(\gamma)=\LenCurve(\gamma_-(t))$, где $\gamma_+(t)$ -- произвольная липшицева параметризация контура $\gamma$ с положительной ориентацией, а $\gamma_-(t)$ -- с отрицательной.

В отличие от риманового случая, финслерова структура не приводит к естественному понятию площади области. Без сомнений, можно предложить различные варианты определения и рассмотреть соответствующие им изопериметрические задачи. В данной работе предлагается использовать следующую форму площади: $\omega=\frac{1}{y^2}dx\wedge dy$. Тогда для области $U\subset G$ ее площадь $\Area(U)$ выражается через интеграл
$$
\Area(U)=\int\limits_U \frac{1}{y^2}dx\wedge dy.
$$
Не утверждая, что выбранное определение является единственно верным, предложим несколько мотивирующих соображений разной степени разумности. Данная форма может быть получена из требования левоинвариантности на $G$ (<<согласованность>> с групповой структурой), что в данном случае также эквивалентно и требованию равенства площадей множеств $\Omega_g$ в каждой точке $g\in G$ (<<согласованность>> с финслеровой нормой). В случае, когда $\Omega$ есть единичный круг, финслерова метрика является римановой и порождает данную форму площади стандартным образом. Также данная форма выделяется среди остальных в следующем смысле. Рассмотрение изопериметрической задачи с формой площади общего вида $w(x,y)dx\wedge dy$ ($w>0$) приводит к системе ОДУ, аналогичной \eqref{pmp-diff-eq-adjoint}: $\dot{h}_1=\psi\dot{y}, \dot{h}_2=-\psi\dot{x}$, где функция $\psi$ является первым интегралом тогда и только тогда, когда $w(x,y)$ пропорциональна $\frac{1}{y^2}$.

Итак, для любой точки $g_{\circ}\in G$ и числа $A_{\circ}> 0$ требуется найти такие простые липшицевы контуры $\gamma_+\subset G$, $\gamma_-\subset G$, что
\begin{equation}
	\label{geom-problem-statement}
	\tag{GP$_{\pm}$}
	\begin{cases}
		\Length_{\pm}(\gamma_\pm) \to \min,\\
		\Area(U_{\gamma_\pm}) = A_{\circ},\\
		g_{\circ}\in \gamma_\pm,
	\end{cases}
\end{equation}
где $U_{\gamma_\pm}$ -- области, ограниченные контурами $\gamma_\pm$ соответственно.

Заметим, что в случае, когда $\Omega$ является евклидовым единичным кругом, каждая из задач \eqref{geom-problem-statement} есть в точности изопериметрическая задача на классической плоскости Лобачевского кривизны $-1$ (для плоскости кривизны $-\frac{1}{R^2}$ нужно взять круг радиуса $\frac{1}{R}$ и умножить форму площади на константу $R^2$).

\section{Переход к задаче оптимального управления}
\label{section_optimal_control_statement}
Следующие преобразования достаточно стандартны и фактически повторяют поиск изопериметрических кривых на евклидовой плоскости.

Для липшицева контура $\gamma_+$ всегда можно выбрать липшицеву натуральную параметризацию $\gamma_{+}(t)=(x(t), y(t)),t\in [0,T]$ с положительной ориентацией. То есть, такую, что $\|\dot{\gamma}_+(t)\|_{\gamma_+(t)}=1$ для п.в. $t$. 
Тогда $\Length_+(\gamma_+)=T$, и $(\dot{x},\dot{y})\in \partial\Omega_{(x,y)}$ для п.в. $t$. Введем управление $(u(t),v(t))=\frac{1}{y}(\dot{x},\dot{y})\in \partial\Omega$ для п.в. $t\in [0,T]$.

Далее, используя формулу Грина для областей с липшицевой границей (см. \cite{EvansGariepy}), выразим площадь области $U_{\gamma_+}$ через параметризацию $\gamma_{+}(t)$ ее границы
\begin{equation}
	\label{area-set-integral}
\Area(U_{\gamma_+})=\int\limits_0^T \frac{\dot{x}(t)}{y(t)}dt=\int\limits_0^T u(t)dt.
\end{equation}
Стандартным образом введем вспомогательную липшицеву переменную $z(t)$, соответствующую изменению площади: $\dot{z}(t)=u(t)$ п.в., $z(0)=0$, $z(T)=A_{\circ}$. Выбор один-формы $\alpha$ со свойством $d\alpha=\omega$ в формуле Грина произволен, но влияет на переменную $z$. В \eqref{area-set-integral} мы выбрали форму $\frac{1}{y}dx$ наиболее простого вида, которая, в частности, левоинвариантна.

Наконец, выполним релаксацию, переходя от $\partial\Omega$ к выпуклому множеству $\Omega$. 

Заметим, что для контура $\gamma_-$ необходимо выбрать натуральную параметризацию с отрицательной ориентацией, что приведет к минусу в формуле Грина и потому к условию $z(T)=-A_{\circ}$. Это будет единственное отличие в формулировках двух задач оптимального управления, поэтому будет удобно объединить их в одну.

Получаем следующую задачу быстродействия
\begin{equation}
	\tag{CP}
	\label{optimal-control-problem-statement}
	\begin{cases}
		T\to\min,
		\\ \dot{x}(t)=y(t)u(t) \text{ для п.в. } t\in [0,T],
		\\
		\dot{y}(t)=y(t)v(t)\text{ для п.в. } t\in [0,T],
		\\
		\dot{z}(t)=u(t)\text{ для п.в. } t\in [0,T],
		\\
		(u(t),v(t))\in \Omega \text{ для п.в. } t\in [0,T],
		\\
		(x(0),y(0))=(x(T),y(T))=(x_{\circ},y_{\circ}),y_{\circ}>0,
		\\
		z(0)=0, z(T)=F_{\circ}\ne 0,
	\end{cases}
\end{equation}
где $x(t),y(t),z(t)$ - липшицевы, $u(t),v(t)\in L_{\infty}$, $t\in [0,T]$, $T>0$, $|F_{\circ}|=A_{\circ}$.

В постановке \eqref{optimal-control-problem-statement} неявно сохраняется условие $y>0$, так как любая допустимая траектория не может покинуть верхнюю полуплоскость. Однако потеряно условие $\gamma_{\pm}(t)\ne\gamma_{\pm}(s)$ при $t,s\in (0,T)$, $t\ne s$. Потому функционал $T$ минимизируется на большем множестве допустимых траекторий, чем следовало бы. Тем не менее, если минимум достигается на кривой без самопересечений (что будет показано ниже), то соответствующий контур является оптимальным в соответствующей задаче \eqref{geom-problem-statement}.

Задача \eqref{optimal-control-problem-statement} удовлетворяет всем условиям теоремы Филиппова о существовании оптимального решения задачи быстродействия (существование допустимых траекторий будет доказано ниже явным их нахождением). Таким образом, оптимальная траектория существует и должна быть среди экстремалей, удовлетворяющих принципу максимума Понтрягина.

\begin{remark}
	Задача \eqref{optimal-control-problem-statement} может быть рассмотрена как левоинвариантная задача поиска субфинслеровых геодезических на трехмерной группе Ли $\tilde{G}=\{(x,y,z),$ $ y>0\}$ с групповой операцией $(x_1,y_1,z_1)\cdot (x_2,y_2,z_2)=(x_1+x_2y_1,y_1y_2,z_1+z_2)$. Субфинслерова структура определяется двумерным распределением $\langle \xi=y\partialdersing{x}+\partialdersing{z},\eta=y\partialdersing{y}\rangle\subset T_{(x,y,z)}\tilde{G}$ с почти нормой $\rho$, заданной в каждой точке равенством $\rho(a\xi+b\eta)=\|(a,b)\|_e$, где $a,b\in\R$.
\end{remark}

\section{Принцип максимума Понтрягина}
\label{section-pmp}
Согласно принципу максимума Понтрягина для задачи быстродействия, если $x(t),y(t),z(t),u(t),v(t)$ -- оптимальный процесс в задаче \eqref{optimal-control-problem-statement}, то существуют такие липшицевы функции $p(t),q(t),r(t)$, не равные тождественному нулю одновременно, что для функции Понтрягина $\mathscr{H}$, имеющей вид
$$
\mathscr{H}(x,y,z,p,q,r,u,v)=uh_{1}+vh_{2},
$$ 
где $h_{1}=py+r$ и $h_{2}=qy$, выполнено для п.в. $t\in [0,T]$
$$
\begin{gathered}
(u(t),v(t))\in \argmax\limits_{(w_1,w_2)\in\Omega}\mathscr{H}(x(t),y(t),z(t),p(t),q(t),r(t),w_1,w_2),\\
\mathscr{H}(x(t),y(t),z(t),p(t),q(t),r(t),u(t),v(t))\equiv H_{\circ}\ge 0,
\end{gathered}
$$а также
$$
\dot{p}(t)=-\partialder{\mathscr{H}}{x}=0,\quad
\dot{q}(t)=-\partialder{\mathscr{H}}{y}=-pu-qv,\quad \dot{r}(t)=-\partialder{\mathscr{H}}{z}=0.
$$
Таким образом, $p(t)\equiv p_{\circ}=\const$, $r(t)\equiv r_{\circ}=\const$. Для функций $h_1, h_2$ получаем систему
\begin{equation}
	\label{pmp-diff-eq-adjoint}
	\begin{cases}
		   \dot{h}_{1}=p_{\circ}\dot{y},
		\\ \dot{h}_{2}=-p_{\circ}\dot{x}.
	\end{cases}
\end{equation} 

\textbf{Случай} $H_{\circ}=0$. Из условия максимизации $\mathcal{H}$ следует, что $h_1=h_2\equiv 0$, откуда, учитывая невырожденность сопряженного множителя, можно получить только тождественную траекторию $x\equiv x_{\circ}, y\equiv y_{\circ}, z\equiv 0$, которая не допустима.

\textbf{Случай} $H_{\circ}>0$. Анализ данного случая использует аппарат выпуклой тригонометрии и приемы, подробно изложенные в работах \cite{convexTrig} и \cite{lobachevsky-geodesics}. Мы приведем лишь некоторые необходимые определения и факты, отсылая заинтересованного читателя к указанным выше работам.

При $H_{\circ}>0$, в силу принципа максимума, управление $(u,v)$ обязано принадлежать границе множества $\Omega$, а поэтому точка $(h_1,h_2)$ двигается по границе растянутой в $H_{\circ}$ раз поляры $\Omega^{\circ}$ множества $\Omega$. Этот факт можно удобно сформулировать, используя обобщенные тригонометрические функции $\cos_{\Omega}$ и  $\sin_{\Omega}$, предложенные впервые в \cite{convexTrig}. Данные функции являются липшицевыми и периодическими с периодом $2S_{\Omega}$, где $S_{\Omega}$ есть евклидова площадь множества $\Omega$. При этом, замкнутая кривая $(\cos_{\Omega}(\theta), \sin_{\Omega}(\theta))$, $\theta\in[0,2S_{\Omega}]$ параметризует границу $\partial\Omega$ множества $\Omega$, имеет положительную ориентацию и начальную точку на положительной оси абсцисс: $\sin_{\Omega}(0)=0$, $\cos_{\Omega}(0)>0$. В случае, когда $\Omega$ есть единичный круг, функции $\cos_{\Omega},\sin_{\Omega}$ совпадают с обычными тригонометрическими функциями. Параметр $\theta$ будем называть по аналогии с классическим случаем обобщенным углом. Также нам понадобятся и функции $\cos_{\Omega^{\circ}}, \sin_{\Omega^{\circ}}$ с аналогичными свойствами, определенные для поляры $\polarSet$ и параметризующие ее границу. Евклидову площадь поляры будем обозначать аналогично через $S_{\polarSet}$.

Важным является понятие соответствующих (относительно $\Omega$) углов. Обобщенные углы $\theta$ и $\theta^{\circ}$ называются соответствующими ($\theta\leftrightarrow\theta^{\circ}$), если выполнено обобщенное тригонометрическое тождество: 
\begin{equation}
\label{pythagorean-identity}
\theta\leftrightarrow\theta^{\circ}\quad\Longleftrightarrow\quad\cos_{\Omega}(\theta)\cos_{\Omega^{\circ}}(\theta^{\circ})+\sin_{\Omega}(\theta)\sin_{\Omega^{\circ}}(\theta^{\circ})=1.
\end{equation}
Таким образом определено многозначное отображение $\theta(\theta^{\circ})$, ставящее в соответствие каждому углу $\theta^{\circ}$ отрезок (с точностью до периода $2S_{\set}$) соответствующих ему углов. Данный отрезок невырожден, только если $\theta^{\circ}$ -- точка недифференцируемости границы поляры, а потому
для п.в. $\theta^{\circ}$ множество $\theta(\theta^{\circ})$ состоит из одного элемента. Далее будем считать (монотонную) функцию $\theta(\theta^{\circ})$ однозначной, подразумевая, что в точках негладкости $\partial\polarSet$ выбран любой соответствующий угол из образа.

Среди всех возможных параметризаций границ $\partial\set$, $\partial\polarSet$ параметризация обобщенными тригонометрическими функциями удовлетворяет наиболее простой (в определенном смысле) системе ОДУ. А именно, для п.в. $\theta^{\circ}$ выполнено
\begin{equation}
	\label{conv-trig-DE-formulae}
	\cos'_{\Omega^{\circ}}(\theta^{\circ})=-\sin_{\Omega}(\theta(\theta^{\circ})),\quad \sin'_{\Omega^{\circ}}(\theta^{\circ})=\cos_{\Omega}(\theta(\theta^{\circ})),
\end{equation}
что по сути и является определением данной параметризации (с учетом начальных значений).

Используя описанные выше инструменты, получаем, что для некоторых функций $\theta^{\circ}(t)$, $\theta(t)$ должно быть выполнено
$$
\label{pmp-h1-h2-formulae}
\begin{cases}
	h_1(t)=H_{\circ}\cos_{\Omega^{\circ}}(\theta^{\circ}(t)),
	\\h_2(t)=H_{\circ}\sin_{\Omega^{\circ}}(\theta^{\circ}(t)),
	\\u(t)=\cos_{\Omega}(\theta(t)),
	\\v(t)=\sin_{\Omega}(\theta(t)),
	\\ \theta(t)\leftrightarrow \theta^{\circ}(t).
\end{cases}
$$
В работе \cite{convexTrig} представлены формулы обобщенной полярной замены координат, откуда следует, что для липшицевых функций $h_1(t),h_2(t)$, не обращающихся одновременно в ноль, функция $\theta^{\circ}(t)$ является липшицевой и удовлетворяет п.в. уравнению
\begin{equation}
\label{pmp-polar-ang-dot}
\dot{\theta^{\circ}}=\frac{1}{H_{\circ}^2}(\dot{h}_2h_1-\dot{h}_1h_2)=-\frac{p_{\circ}}{H_{\circ}}y.
\end{equation}
Функция $\theta(t)$ будет восстанавливаться по $\theta^{\circ}(t)$ из соотношения соответствия. Если граница поляры $\Omega^{\circ}$ является $C^{1}$ гладкой, функция $\theta(t)$ восстанавливается единственным образом и является непрерывной. В общем случае значение $\theta(t)$ восстанавливается однозначно для п.в. $t$, а функция $\theta(t)$ измерима. Учитывая всё вышесказанное, можно записать $\theta(t)=\theta(\theta^{\circ}(t))$.

\textbf{Случай} $H_{\circ}>0$, $r_{\circ}=0$. В этом случае условие максимума теряет связь с переменной $z(t)$, что приводит к самодостаточной задаче поиска геодезических $(x(t),y(t))$ на финслеровой плоскости Лобачевского, которая была полностью решена в работе \cite{lobachevsky-geodesics}. Не вдаваясь в подробности, дадим несколько пояснений. При $p_{\circ}=0$ геодезические называются <<вертикальными>> (аналоги вертикальных прямых для классической плоскости Лобачевского). В этом случае управление $(u,v)$ почти всюду принадлежит одной из двух горизонтальных граней (возможно вырождающихся в точку) границы $\partial\Omega$, где $\sin_{\Omega}$ достигает максимального или минимального значения. Таким образом, мы получаем либо линейные траектории c постоянным управлением (если грань есть точка), либо особые по грани траектории. При $p_{\circ}\ne 0$ геодезические называются <<горизонтальными>> (аналоги верхних половин окружностей в классическом случае) и представляют собой верхние <<половины>> границы поляры, растянутой, повернутой на $\pm\frac{\pi}{2}$ в зависимости от знака $p_{\circ}$ и сдвинутой горизонтально.

Никакие геодезические не являются допустимыми экстремалями для нашей задачи, так как не могут быть замкнутыми. Неочевидным является лишь случай особых по грани геодезических. Для последующих рассуждений докажем более общее утверждение: никакая особая по грани траектория не может быть замкнутой.

Действительно, пусть $(x(t),y(t)),t\in[0,T]$ -- особая по грани с вершинами $(u_1,v_1)$ и $(u_2,v_2)$ траектория системы. Так как управление $(u(t),v(t))$, соответствующее данной траектории, почти всюду лежит в выпуклом секторе между вершинами грани, а касательный вектор $(\dot{x}(t),\dot{y}(t))$ почти всюду сонаправлен с вектором $(u(t),v(t))$, то для любого момента времени $s\in[0,T)$ часть кривой $(x(t),y(t))$ при $t\in(s,T]$ лежит в выпуклом секторе между линейными траекториями системы, соответствующими постоянным управлениям $(u_1,v_1)$ и $(u_2,v_2)$ и выходящими из точки $(x(s),y(s))$, что противоречит замкнутости траектории $(x(t),y(t))$.

\textbf{Случай} $H_{\circ}>0$, $r_{\circ}\ne 0$, $p_{\circ}=0$. В этом случае $\theta^{\circ}\equiv \theta^{\circ}(0)$, $h_1\equiv r_{\circ}\ne 0$, $h_2=\const$.
\\Если множество $\theta(\theta^{\circ}(0))$ состоит только из одного элемента, мы получаем $u(t)\equiv u(0)$, $v(t)\equiv v(0)$, причем $v(0)$ больше, чем $\min(\sin_{\Omega})$, и меньше, чем $\max(\sin_{\Omega})$, а траектория $x(t),y(t)$ линейная (для вертикальных геодезических $v(0)$, наоборот, обязано быть экстремумом). Если же множество $\theta(\theta^{\circ}(0))$ есть невырожденный отрезок, то измеримое управление $u(t),v(t)$ может быть выбрано многими способами с условием принадлежности определенной негоризонтальной грани множества $\Omega$. Соответствующие траектории $x(t),y(t)$ являются особыми по этой грани и расположены в выпуклом секторе между линейными траекториями, полученными при выборе управления, тождественно равного вершинам этой грани. При разборе предыдущего случая было показано, что особые по грани траектории не могут быть замкнуты. Поэтому перечисленные траектории также не являются допустимыми экстремалями для нашей задачи. 

\textbf{Основной случай} $H_{\circ}>0$, $r_{\circ}\ne 0$, $p_{\circ}\ne 0$. В этом случае функции $x,y$ могут быть легко найдены из системы \eqref{pmp-diff-eq-adjoint}:
$$
\begin{cases}
	x(t)=-\frac{1}{p_{\circ}}h_2(t)+c_{x\circ},\\
	y(t)=\frac{1}{p_{\circ}}h_1(t)+c_{y\circ},
\end{cases}
$$
где $c_{x\circ},c_{y\circ}$ -- константы, причем $c_{y\circ}=-\frac{r_{\circ}}{p_{\circ}}\ne 0$.

Еще раз заметим, что выбор любой другой формы площади значительно усложнил бы (или сделал бы интереснее) интегрирование системы в этом месте повествования.

Вводя обозначение $R_{\circ}=-\frac{H_{\circ}}{p_{\circ}}\ne 0$ и используя \eqref{pmp-polar-ang-dot}, окончательно получаем
\begin{equation}
\label{pmp-xy-formulae}
\begin{cases}
	x(t)=R_{\circ}\sin_{\polarSet}(\theta^{\circ}(t))+c_{x\circ},\\
	y(t)=-R_{\circ}\cos_{\polarSet}(\theta^{\circ}(t))+c_{y\circ},
	\\ \dot{\theta}^{\circ}(t)=\frac{y(t)}{R_{\circ}}=\frac{c_{y\circ}}{R_{\circ}}-\cos_{\polarSet}(\theta^{\circ}(t)).
\end{cases}
\end{equation}
Выписанные экстремали являются единственными претендентами на роль оптимальных в задаче \eqref{optimal-control-problem-statement}.

\section{Дальнейшее исследование экстремалей}
Из системы \eqref{pmp-xy-formulae} видно, что точка $(x,y)$ двигается по границе поляры $\polarSet$, повернутой на $\mp\frac{\pi}{2}$ (в зависимости от знака $R_{\circ}$), растянутой в $|R_{\circ}|$ раз и сдвинутой на вектор $(c_{x\circ},c_{y\circ})$. При этом, для замкнутости кривой точка $(x,y)$ должна совершить минимум один оборот по границе. Так как при следующих оборотах длина и площадь будут суммироваться, рассмотрим сперва, как связаны длина и площадь на экстремалях, полученных при разовом обороте.

\subsection{Функции $L_{\pm}(\lambda), F_{\pm}(\lambda)$}
Заметим, что так как параметризация границы тригонометрическими функциями имеет положительную ориентацию по определению, то ориентация замкнутых кривых \eqref{pmp-xy-formulae} положительна при $R_{\circ}>0$ и отрицательна при $R_{\circ}<0$.

Введем следующие обозначения:
\begin{equation}
\label{lambda-def}
	\begin{gathered}
	\lambda_{\circ}=\frac{c_{y\circ}}{R_{\circ}},\quad-M_{-}^{\circ}=\min_{[0,2S_{\polarSet}]}\cos_{\Omega^{\circ}}(\theta^{\circ})<0,\quad M_{+}^{\circ}=\max_{[0,2S_{\polarSet}]}\cos_{\Omega^{\circ}}(\theta^{\circ})>0.
	\end{gathered}
\end{equation}
Для того чтобы траектория \eqref{pmp-xy-formulae} при $R_{\circ}>0$ описывала простой замкнутый контур на $G$, необходимо и достаточно, чтобы угол $\theta^{\circ}$ совершал за время $T$ полный и только один оборот от $\theta^{\circ}(0)$ до $\theta^{\circ}(0)+2S_{\polarSet}$. Так как допустимая траектория не может покинуть верхнюю полуплоскость, выполнено $\lambda_{\circ}>M_{+}^{\circ}$. Аналогично, при $R_{\circ}<0$ необходимо и достаточно, чтобы угол $\theta^{\circ}$ совершал оборот от $\theta^{\circ}(0)$ до $\theta^{\circ}(0)-2S_{\polarSet}$, при этом выполнено $\lambda_{\circ}<-M_{-}^{\circ}$. 

Для выражения длины экстремалей и площади заметаемой ими области осуществим замену переменной $t$ на $\theta^{\circ}$ в соответствующих интегралах и воспользуемся периодичностью тригонометрических функций. Мы получаем, что длина и площадь экстремали зависят лишь от соотношения $\lambda_{\circ}=\frac{c_{y\circ}}{R_{\circ}}$, что приводит к рассмотрению функций $L_+(\lambda), F_+(\lambda)$ при $\lambda\in (M_+^{\circ},+\infty)$ -- значения длины и площади экстремалей с положительной ориентацией, и $L_-(\lambda), F_-(\lambda)$ при $\lambda\in (-\infty,-M_-^{\circ})$ -- значения длины и отрицательной площади экстремалей с отрицательной ориентацией. Непосредственные вычисления приводят к следующим выражениям
\begin{equation}
\begin{gathered}
\label{L-F-of-lambda-formulae-plus}
L_+(\lambda)=\int\limits_{0}^{2S_{\polarSet}}\frac{d\theta^{\circ}}{\lambda-\cos_{\Omega^{\circ}}(\theta^{\circ})},\quad
F_+(\lambda)=\int\limits_{0}^{2S_{\polarSet}}\frac{\cos_{\Omega}(\theta(\theta^{\circ}))d\theta^{\circ}}{\lambda-\cos_{\Omega^{\circ}}(\theta^{\circ})},\quad\lambda>M^{\circ}_+.
\end{gathered}
\end{equation}
\begin{equation}
\begin{gathered}
\label{L-F-of-lambda-formulae-minus}
L_-(\lambda)=\int\limits_{0}^{2S_{\polarSet}}\frac{d\theta^{\circ}}{\cos_{\Omega^{\circ}}(\theta^{\circ})-\lambda},\quad
F_-(\lambda)=\int\limits_{0}^{2S_{\polarSet}}\frac{\cos_{\Omega}(\theta(\theta^{\circ}))d\theta^{\circ}}{\cos_{\Omega^{\circ}}(\theta^{\circ})-\lambda},\quad\lambda<-M^{\circ}_-.
\end{gathered}
\end{equation}
Установим несколько свойств функций $L_{\pm}, F_{\pm}$, напрямую получаемых из интегральных формул \eqref{L-F-of-lambda-formulae-plus}, \eqref{L-F-of-lambda-formulae-minus}. 
Исследование функций $L_-,F_-$ может быть сведено к исследованию $L_+,F_+$ для изопериметрической задачи со множеством $-\set$ или проведено аналогично, поэтому далее будем рассматривать только функции $L_+,F_+$. В частности, если $\Omega = -\Omega$, то $M_{+}^{\circ}=M_{-}^{\circ}$, $L_{+}(\lambda)=L_{-}(-\lambda)$ и $F_{+}(\lambda)=-F_{-}(-\lambda)$ $\forall \lambda>M_{+}^{\circ}$.

\begin{proposition}
\label{L-A-DE}
Функции $L_{+},F_+\in C^{\infty}((M_+^{\circ},+\infty))$, и верно ${L'_{+}(\lambda)=\lambda F'_{+}(\lambda)}$.
\begin{proof}
Пользуясь дифференциальными соотношениями \eqref{conv-trig-DE-formulae}, перейдем к интегралу Лебега-Стильтьеса, для которого ввиду липшицевости тригонометрических функций справедлива формула интегрирования по частям.
$$
F_+(\lambda)=\int\limits_{0}^{2S_{\polarSet}}\frac{d\sin_{\polarSet}(\theta^{\circ})}{\lambda-\cos_{\polarSet}(\theta^{\circ})}=\left.\frac{\sin_{\polarSet}(\theta^{\circ})}{\lambda-\cos_{\polarSet}(\theta^{\circ})}\right\vert_{0}^{2S_{\polarSet}}-\int\limits_{0}^{2S_{\polarSet}}\sin_{\polarSet}(\theta^{\circ})d\left(\frac{1}{\lambda-\cos_{\polarSet}(\theta^{\circ})}\right).
$$
Учитывая периодичность обобщенных тригонометрических функций и снова соотношения \eqref{conv-trig-DE-formulae}, получаем
\begin{equation}
	\label{A-alt-form}
	F_+(\lambda)=\int\limits_{0}^{2S_{\polarSet}}\frac{\sin_{\polarSet}(\theta^{\circ})\sin_{\Omega}(\theta(\theta^{\circ}))}{(\lambda-\cos_{\polarSet}(\theta^{\circ}))^2}d\theta^{\circ}.
\end{equation}
Тогда, вычитая \eqref{A-alt-form} из выражения \eqref{L-F-of-lambda-formulae-plus} для $F_+$, имеем
$$
0=F_+(\lambda)-F_+(\lambda)=\int\limits_{0}^{2S_{\polarSet}}\frac{\cos_{\Omega}(\theta(\theta^{\circ}))(\lambda-\cos_{\polarSet}(\theta^{\circ}))-\sin_{\polarSet}(\theta^{\circ})\sin_{\Omega}(\theta(\theta^{\circ}))}{(\lambda-\cos_{\polarSet}(\theta^{\circ}))^2}d\theta^{\circ}.
$$
Воспользовавшись тригонометрическим тождеством \eqref{pythagorean-identity}, получаем:
$$
0=\lambda\int\limits_{0}^{2S_{\polarSet}}\frac{\cos_{\Omega}(\theta(\theta^{\circ}))}{(\lambda-\cos_{\polarSet}(\theta^{\circ}))^2}d\theta^{\circ}-\int\limits_{0}^{2S_{\polarSet}}\frac{1}{(\lambda-\cos_{\polarSet}(\theta^{\circ}))^2}d\theta^{\circ}=-\lambda F'_+(\lambda)+L'_{+}(\lambda).
$$
\end{proof}
\end{proposition}
\begin{proposition}
\label{L-props}
Пусть $k\in\mathbb{N}\cup\{0\}$. Функция $L_{+}$ обладает следующими свойствами.
\begin{enumerate}
\item \label{L-simple-props} Функция $L_{+}$ строго убывает от $+\infty$ до $0$ и строго выпукла.
\item\label{L-der-limit-infty}
$\lim\limits_{\lambda\to +\infty}\lambda^{k+1} L_{+}^{(k)}(\lambda)=(-1)^{k}k!2S_{\Omega^{\circ}}$.
\item\label{L-der-limit-M0}
Пусть $\argmax(\cos_{\polarSet})=[\theta^{\circ}_1,\theta^{\circ}_2]$, где $\theta^{\circ}_1< \theta^{\circ}_2$. Тогда
$$
\lim\limits_{\lambda\to M^{\circ}_{+}}(\lambda- M^{\circ}_{+})^{k+1}L_{+}^{(k)}(\lambda)=(-1)^kk!(\theta^{\circ}_2-\theta^{\circ}_1).
$$
Пусть $\argmax(\cos_{\polarSet})=\{\theta^{\circ}_1\}$. Тогда
$$
\Lim(L)_{+\nu,k}=\lim\limits_{\lambda\to M^{\circ}_{+}}(\lambda- M^{\circ}_{+})^{\nu}L_{+}^{(k)}(\lambda)=
\begin{cases}
	(-1)^k\infty, \nu\le 0,
	\\(-1)^k\infty, 0<\nu<k,
	\\ \ne 0,\nu=k>0,
	\\0,\nu \ge k+1.
\end{cases}
$$
При $k=0,0<\nu<1$ и $k>0,k\le\nu<k+1$ значение предела зависит от поведения $\cos_{\polarSet}$ в малой окрестности $\theta^{\circ}_1$. В частности, если существуют константы $A>0$, $\alpha\ge 1$, что в некоторой окрестности $\theta_1^{\circ}$ верно $M_+^{\circ}-\cos_{\polarSet}(\theta^{0})\ge A|\theta^{0}-\theta^{0}_1|^\alpha$, то для $\nu>k+1-\frac{1}{\alpha}$ выполнено $\Lim(L)_{+\nu,k}=0$.
\end{enumerate}
\end{proposition}
\begin{proof}
Пункты \ref{L-simple-props}, \ref{L-der-limit-infty} и пункт \ref{L-der-limit-M0} в случае, когда максимум функции $\cos_{\polarSet}$ достигается на невырожденном отрезке, следуют из определения \eqref{L-F-of-lambda-formulae-plus} непосредственно.
 
Докажем пункт \ref{L-der-limit-M0} в случае единственного аргмаксимума. Обозначим $\Delta=\lambda-M^{\circ}_+$, $x=\theta^{\circ}-\theta^{\circ}_1$, $\varphi(x)=M^{\circ}_+-\cos_{\polarSet}(x+\theta^{\circ}_1)$. Зафиксируем малое $\delta>0$ и обозначим $J=[-\theta^{\circ}_1,2S_{\polarSet}-\theta^{\circ}_1]\smallsetminus [-\delta,\delta]$. Тогда
\begin{equation}
\label{L-der-limit-M0-change-var}
\Lim(L)_{+\nu,k}=(-1)^kk!\left(\lim\limits_{\Delta\to 0}\int\limits_{-\delta}^{\delta}\frac{\Delta^{\nu}dx}{(\Delta+\varphi(x))^{k+1}}+\lim\limits_{\Delta\to 0}\int\limits_{J}\frac{\Delta^{\nu}dx}{(\Delta+\varphi(x))^{k+1}}\right),
\end{equation}
причем второй предел в скобке конечен при $\nu=0$, равен нулю при $\nu>0$ и равен $+\infty$ при $\nu<0$. Оценим первый предел.

Пусть $C>0$ -- некоторая константа Липшица для $\varphi$. Тогда
$$
\int\limits_{-\delta}^{\delta}\frac{\Delta^{\nu}}{(\Delta+\varphi(x))^{k+1}}dx\ge 2\int\limits_{0}^{\delta}\frac{\Delta^{\nu-k-1}}{(1+\frac{C}{\Delta}x)^{k+1}}dx=\frac{2\Delta^{\nu-k}}{C}\int\limits_{0}^{C\delta/\Delta}\frac{1}{(1+y)^{k+1}}dy>0.
$$
При $k=0$, неограниченность последнего выражения при $\Delta\to 0$ эквивалентна условию $\nu\le 0$, при $k>0$ -- условию $\nu<k$. При $\nu=k>0$ предел при $\Delta\to 0$ последнего выражения конечен и положителен.

Для $\nu\ge k+1$ утверждение следует из неотрицательности $\varphi$.

Предположим теперь, что в некоторой $\eps$-окрестности точки $x=0$ верно $\varphi(x)\ge A|x|^{\alpha}$. Тогда, выбрав в \eqref{L-der-limit-M0-change-var} $\delta$ меньше $\eps$, получаем
$$
\int\limits_{-\delta}^{\delta}\frac{\Delta^{\nu}}{(\Delta+\varphi(x))^{k+1}}dx\le2\int\limits_{0}^{\delta}\frac{\Delta^{\nu-k-1}}{(1+\frac{A}{\Delta}x^{\alpha})^{k+1}}dx=\frac{2\Delta^{\nu-k-1+\frac{1}{\alpha}}}{A^{\frac{1}{\alpha}}}\int\limits_{0}^{\sqrt[\alpha]{A}\delta/\sqrt[\alpha]{\Delta}}\frac{1}{(1+y^{\alpha})^{k+1}}dy.
$$
При $k=0,\alpha=1$ последний интеграл считается явно и зависит от $\Delta$ логарифмически. При $k>0$ или $\alpha>1$ последний интеграл можно оценить сходящимся интегралом с бесконечным верхним пределом. Поэтому, если $\nu-k-1+\frac{1}{\alpha}>0$, последнее выражение стремится к нулю при $\Delta\to 0$.
\end{proof}
\begin{proposition}
\label{A-props}
	Пусть $k\in\mathbb{N}\cup\{0\}$. Функция $F_+$ обладает следующими свойствами.
	\begin{enumerate}
		\item \label{A-simple-props} Функция $F_+$ строго убывает от $+\infty$ до $0$ и строго выпукла.
		\item\label{A-der-limit-infty}
		$\lim\limits_{\lambda\to +\infty}\lambda^{k+2} F_+^{(k)}(\lambda)=(-1)^k(k+1)!S_{\polarSet}$.
		\item\label{L-A-integral-formula}
		$F_+(\lambda)=-\int\limits_{\lambda}^{+\infty}\frac{L'_+(\mu)}{\mu}d\mu=\frac{L_+(\lambda)}{\lambda}-\int\limits_{\lambda}^{+\infty}\frac{L_+(\mu)}{\mu^2}d\mu$.
		\item\label{A-der-limit-M0}
		$\lim\limits_{\lambda\to M^{\circ}_{+}}(\lambda- M^{\circ}_{+})^{\nu}F_+^{(k)}(\lambda)=\frac{1}{M^{\circ}_+}\lim\limits_{\lambda\to M^{\circ}_{+}}(\lambda- M^{\circ}_{+})^{\nu}L_{+}^{(k)}(\lambda)$, $\nu\in\R$.
	\end{enumerate}
\end{proposition}
\begin{proof}
Все свойства функции $F_+$ при $k\ge 1$ выводятся из соответствующих свойств $L_+$ и предложения \ref{L-A-DE}. Покажем, что утверждения верны при $k=0$.

Докажем пункт \ref{A-der-limit-infty}. Обозначим
$$
I_s=\int\limits_0^{2S_{\Omega^{\circ}}}\sin_{\Omega}(\theta(\theta^{\circ}))\sin_{\Omega^{\circ}}(\theta^{\circ})d\theta^{\circ},\quad
I_c=\int\limits_0^{2S_{\Omega^{\circ}}}\cos_{\Omega}(\theta(\theta^{\circ}))\cos_{\Omega^{\circ}}(\theta^{\circ})d\theta^{\circ}.
$$
Из соотношения \eqref{A-alt-form} получаем, $\lambda^2 F_+\to I_s, \lambda\to+\infty$. Из интегрирования по частям следует, что $I_s=I_c$. С другой стороны, ввиду тригонометрического тождества, $I_c+I_s=2S_{\Omega^{\circ}}$, откуда $I_s=S_{\Omega^{\circ}}$.

Из предложения \ref{L-A-DE} и полученного предела $\lim_{\lambda\to +\infty} F_+(\lambda)=0$ получаем интегральное соотношение в пункте \ref{L-A-integral-formula}, которое корректно в силу асимптотики $L_+$ при $\lambda\to +\infty$. Отсюда, в частности, следует пункт \ref{A-simple-props}.

Далее пункт \ref{A-der-limit-M0} также следует из интегрального соотношения в пункте \ref{L-A-integral-formula}.
\end{proof}

\begin{corollary}
\label{A-lambda-equation}
Пусть $A_{\circ}>0$, тогда уравнение $F_+(\lambda)=A_{\circ}$ имеет и единственное решение $\lambda_{\circ}\in (M_+^{\circ},+\infty)$.
\end{corollary}

\subsection{Функции $\mathcal{L}_{\pm}(F), \mathcal{F}_{\pm}(L)$}
Монотонность функций $L_+(\lambda), F_+(\lambda)$ позволяет рассмотреть зависимость длины от площади и площади от длины.

Пусть $\Lambda_+(L),L>0$ есть функция, обратная к $L_+$. Рассмотрим функцию $\mathcal{F_+}(L)=F_+(\Lambda_+(L)), L>0$. Из предложений \ref{L-A-DE}, \ref{L-props}, \ref{A-props} следует
\begin{proposition}
\label{F-of-L-props}
	Функция $\mathcal{F_+}$ обладает следующими свойствами.
	\begin{enumerate}
		\item Функция $\mathcal{F_+}\in C^{\infty}((0,+\infty))$, строго возрастает от $0$ до $+\infty$ и строго выпукла.
		\item $\mathcal{F_+}'(L)=\frac{1}{\Lambda_+(L)}$, $\lim\limits_{L\to 0}\mathcal{F_+}'(L)=0$, $\lim\limits_{L\to +\infty}\mathcal{F_+}'(L)=\frac{1}{M_{+}^{\circ}}$.
		\item $\mathcal{F_+}''(L)=-\frac{\Lambda'_+(L)}{(\Lambda_+(L))^2}$, $\lim\limits_{L\to 0}\mathcal{F_+}''(L)=\frac{1}{2S_{\polarSet}}$, $\lim\limits_{L\to +\infty}\mathcal{F_+}''(L)=0$.
		\item $\lim\limits_{L\to 0}\frac{\mathcal{F_+}(L)}{L\mathcal{F'_+}(L)}=\frac{1}{2}$, $\lim\limits_{L\to 0}\frac{\mathcal{F_+}(L)}{L^2}=\frac{1}{4S_{\polarSet}}$, $\lim\limits_{L\to +\infty}\frac{\mathcal{F_+}(L)}{L}=\frac{1}{M_+^{\circ}}$.
		\item 
		Существование асимптоты кривой $(L,\mathcal{F_+}(L)),L>0$ при $L\to +\infty$ эквивалентно конечности предела
		$$
		\label{A-props-asymptote-a}
		0<a_+=\lim\limits_{\lambda\to M^{\circ}_+}\int\limits_{\lambda}^{+\infty}\frac{L_+(\mu)}{\mu^2}d\mu=\lim\limits_{\lambda\to M^{\circ}_+}\left(\frac{L_+(\lambda)}{\lambda}-F_+(\lambda)\right)=\lim\limits_{P\to \frac{1}{M^{\circ}_+}}\mathcal{F_+^*}(P),
		$$
		где $\mathcal{F}_+^*(P)=PL_+(\frac{1}{P})-F_+(\frac{1}{P}),P\in (0,\frac{1}{M^{\circ}_+})$ -- функция, двойственная к функции $\mathcal{F}_+$. В случае существования асимптота имеет вид $y(x)=\frac{1}{M_+^{\circ}}x-a_+$.
	\end{enumerate}
\end{proposition}

Монотонность функции $\mathcal{F_+}$ позволяет рассмотреть также обратную к ней строго вогнутую функцию $\mathcal{L_+}(A), A>0$. Аналогично введем и функции $\mathcal{F}_-(L)=F_-(\Lambda_-(L)), L>0$, $\mathcal{L_-}=\mathcal{F}_-^{-1}$, где $\Lambda_-=L_-^{-1}$.

\section{Об изопериметрических контурах}
В нижеследующей теореме дано полное описание изопериметрических контуров на финслеровой плоскости Лобачевского, то есть таких простых (замкнутых) контуров, на которых достигается оптимальное соотношение длины и площади.
\begin{theorem}[Об изопериметрических контурах]
\label{theorem-optimal-curves}
Для любой точки $g_{\circ}=(x_{\circ},y_{\circ})\in G$ и числа $A_{\circ}> 0$ существуют два семейства $\Gamma_+(g_{\circ},A_{\circ})=\{\gamma_{+}^{\alpha}\}$,  $\Gamma_-(g_{\circ},A_{\circ})=\{\gamma_{-}^{\alpha}\}$, $\alpha\in [0,2S_{\polarSet})$ простых липшицевых контуров, являющихся соответственно знаку решениями изопериметрических задач \eqref{geom-problem-statement}. Обратно, любое решение $\gamma_{\pm}$ задачи \eqref{geom-problem-statement} есть контур из семейства $\Gamma_{\pm}(g_{\circ},A_{\circ})$.
	
Каждый контур $\gamma_{\pm}^{\alpha}$ семейства $\Gamma_{\pm}(g_{\circ},A_{\circ})$ является границей поляры $\polarSet$, повернутой на $\mp\frac{\pi}{2}$, растянутой в $|R_{\pm}|$ раз и сдвинутой на вектор $(c_{x\pm},c_{y\pm})$. Соответствующие константы определяются однозначно из системы
	\begin{equation}
		\label{iso-optimal-constants}
		\begin{cases}
			F_{\pm}(\lambda_{\pm})=\pm A_{\circ},\\
			R_{\pm}(\alpha,y_{\circ},A_{\circ})=\frac{y_{\circ}}{\lambda_{\pm}-\cos_{\Omega^{\circ}}(\alpha)},\\
			c_{ x\pm}(\alpha,x_{\circ},y_{\circ},A_{\circ})=x_{\circ}-R_{\pm}\sin_{\polarSet}(\alpha),\\
			c_{ y\pm}(\alpha,y_{\circ},A_{\circ})=R_{\pm}\lambda_{\pm}.
		\end{cases}
	\end{equation}
Обратно, любой контур $\gamma\subset G$, представляющий собой границу поляры, повернутой на $\mp\frac{\pi}{2}$, растянутой в $R>0$ раз и сдвинутой в верхнюю полуплоскость, является изопериметрическим контуром, то есть принадлежит семейству $\Gamma_{\pm}(g_{\circ}, A_{\circ})$ для $g_{\circ}\in \gamma$, $A_{\circ}=\Area(U_{\gamma})$, где $U_{\gamma}$ -- область, ограниченная $\gamma$.

Липшицева натуральная параметризация $(x_\pm(t), y_\pm(t))$ контура $\gamma_{\pm}^{\alpha}$ c положительной ориентацией в случае $\gamma_{+}^{\alpha}$ и отрицательной ориентацией в случае $\gamma_{-}^{\alpha}$ имеет вид
\begin{equation}
\label{iso-curves}
	\begin{cases}
		x_{\pm}(t)=R_{\pm}\sin_{\polarSet}(\theta_{\pm}^{\circ}(t))+c_{x\pm},\\
		y_{\pm}(t)=-R_{\pm}\cos_{\polarSet}(\theta_{\pm}^{\circ}(t))+c_{y\pm},
		\\ \dot{\theta}_{\pm}^{\circ}(t)=\lambda_{\pm}-\cos_{\polarSet}(\theta_{\pm}^{\circ}(t)),\quad{\theta}_{\pm}^{\circ}(0)=\alpha,\\
		0\le t\le L_{\pm}(\lambda_{\pm}).
	\end{cases}
\end{equation}
Здесь $S_{\polarSet}$ есть евклидова площадь поляры $\polarSet$, а функции $L_{\pm}$, $F_{\pm}$ задаются соотношениями \eqref{L-F-of-lambda-formulae-plus}, \eqref{L-F-of-lambda-formulae-minus}.
\end{theorem}
\begin{proof}
Как было замечено в разделе \ref{section_optimal_control_statement}, решение задачи \eqref{optimal-control-problem-statement} существует и является экстремалью принципа максимума Понтрягина. Единственными допустимыми экстремалями являются кривые \eqref{pmp-xy-formulae}, совершающие $k\in \mathbb{N}$ число оборотов по границе поляры. Из строгой вогнутости функции $\mathcal{L}_+$ и предельного соотношения $\mathcal{L}_+(A)\to 0,~A\to 0$ следует, что при $k> 1, A>0$ верно $\mathcal{L}_+(kA)<k\mathcal{L}_+(A)$ (аналогично для $\mathcal{L}_-$). Таким образом, экстремали, полученные более чем одним обходом по границе не являются оптимальными. Из следствия \ref{A-lambda-equation} следует, что при фиксированных $g_{\circ}, A_{\circ}$ для каждого $\alpha$ система \eqref{iso-optimal-constants} разрешается однозначно, что приводит к двум семействам кривых, параметризованных углом $\alpha$. При этом, все кривые из одного семейства имеют равную длину. Таким образом, кривые вида \eqref{iso-curves} и только они являются решениями задачи \eqref{optimal-control-problem-statement}.

Так как минимум в задаче \eqref{optimal-control-problem-statement} достигается на кривых без самопересечения (и только на них), соответствующие образы этих кривых, т.е. контуры семейств $\Gamma_+,\Gamma_-$, и только они являются решениями соответствующих изопериметрических задач \eqref{geom-problem-statement}.
\end{proof}

Заметим, что полученный результат согласуется с известным результатом для классической плоскости Лобачевского о том, что изопериметрические контуры суть евклидовы окружности. Известно, что римановы окружности на классической плоскости Лобачевского тоже суть евклидовы окружности, но со смещенным по вертикали центром. Однако стоит отметить, что, вообще говоря, контуры семейств $\Gamma_{\pm}$ финслеровыми окружностями не являются.

Теорему \ref{theorem-optimal-curves} интересно сравнить c результатами работы \cite{Busemann1947TheIP}, где было доказано, что изопериметрические контуры для евклидовой площади на плоскости Минковского тоже суть границы поляры, повернутой на $\mp\frac{\pi}{2}$, растянутой и сдвинутой. В этом случае длина контура и площадь заметаемой им области не зависят от сдвигов, а потому все изопериметрические контуры с фиксированной площадью и минимизирующие длину в положительном направлении (аналогично для контуров, минимизирующих длину в отрицательном направлении) отличаются лишь сдвигом, что не так в случае финслеровой плоскости Лобачевского.

\begin{figure}[h!]
	\begin{minipage}[h1]{0.32\linewidth}
		\center{\includegraphics[width=1\linewidth]{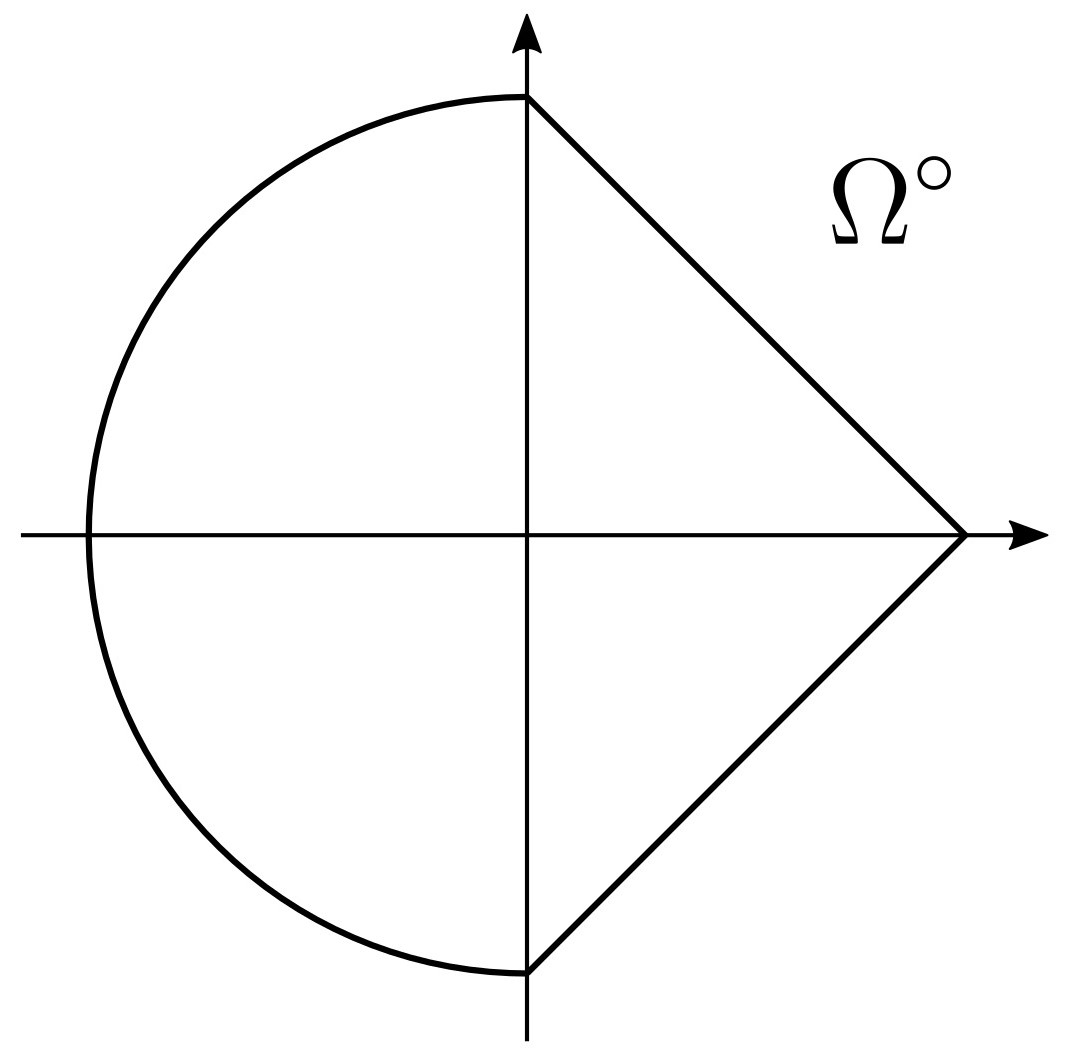}}
	\end{minipage}
	\hfill
	\begin{minipage}[h]{0.66\linewidth}
		\center{\includegraphics[width=1\linewidth]{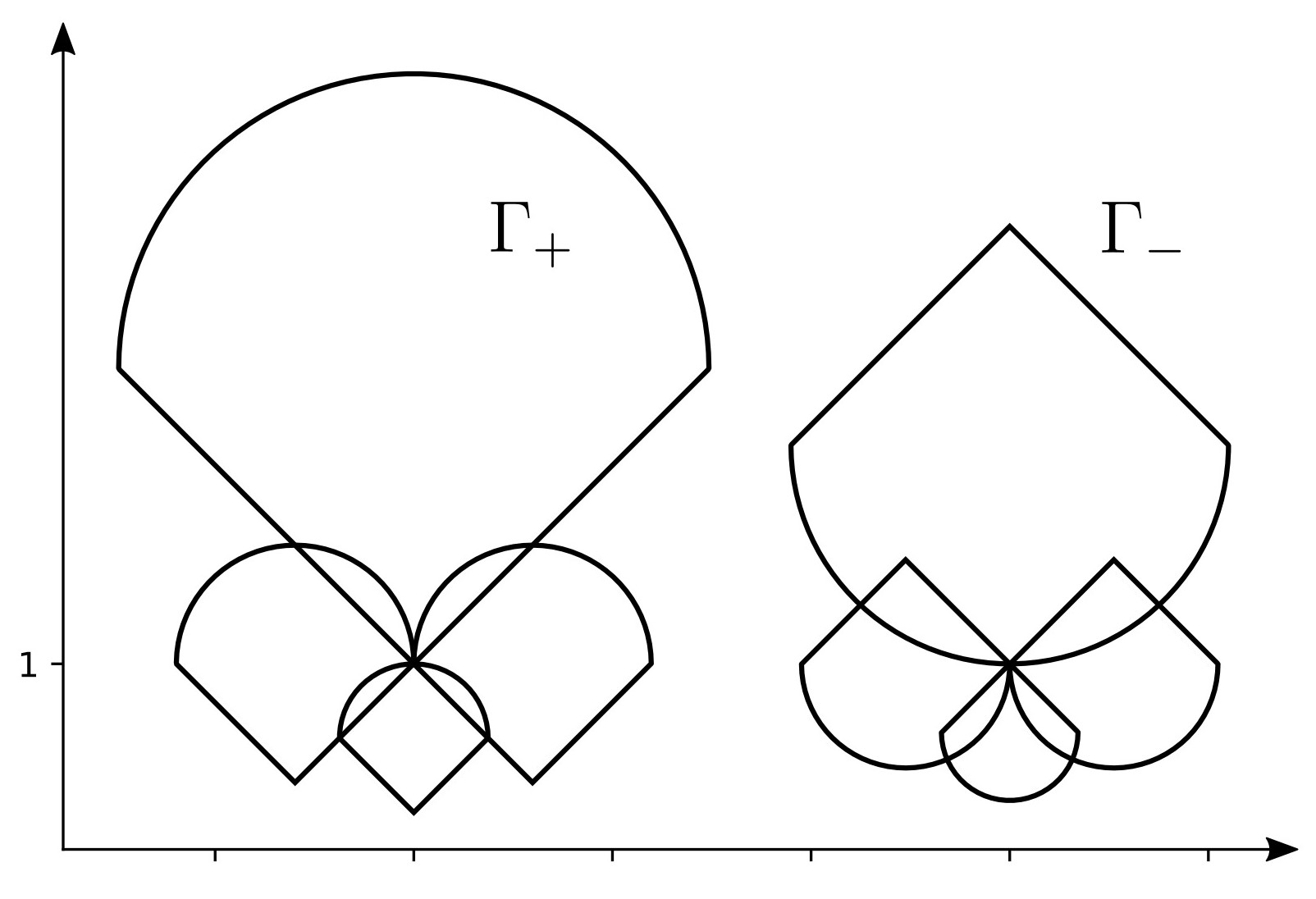}} 
	\end{minipage}
	\caption{Некоторые представители семейств $\Gamma_+((*,1),1)$ и $\Gamma_-((*,1),1)$ при несимметричном множестве $\Omega^{\circ}$.}
\end{figure}
\section{Об изопериметрических соотношениях}
Хорошо известно (см., например, \cite{Burago}), что для простого контура $\gamma$ длины $L_{\circ}$ с площадью заметаемой области $A_{\circ}$ на классической плоскости Лобачевского кривизны $-1$ выполняется неравенство
$$
L_{\circ}^2-4\pi A_{\circ}-A_{\circ}^2\ge 0,
$$
причем равенство достигается на и только на изопериметрических контурах (евклидовых окружностях). Данное неравенство разделяет множество пар $\{(L_{\circ}, A_{\circ})\in \mathbb{R}^2, L_{\circ}>0, A_{\circ}>0\}$ на две области, общей границей которых является гипербола
$$
\label{iso-circle-hyperbola}
\left(\frac{A_{\circ}+2\pi}{2\pi}\right)^2-\left(\frac{L_{\circ}}{2\pi}\right)^2=1.
$$
В общем случае граничная кривая может быть задана в терминах функций $\mathcal{L}_{\pm}$, $\mathcal{F}_{\pm}$.
\begin{theorem}[Об изометрических соотношениях]
\label{theorem-iso-ineq}
Пусть $\gamma\subset G$ - простой липшицев контур, ограничивающий область $U_{\gamma}$. Пусть $\Length_+(\gamma)=L_{+\circ}$, $\Length_-(\gamma)=L_{-\circ}$, $\Area(U_\gamma)=A_{\circ}$. Тогда выполнены следующие изопериметрические соотношения
\begin{equation}
\label{iso-ineq}
\begin{gathered}
L_{+\circ}\ge \mathcal{L_+}(A_{\circ})\quad\left(\Longleftrightarrow A_{\circ}\le \mathcal{F}_+(L_{+\circ})\right),\\
L_{-\circ}\ge \mathcal{L_-}(-A_{\circ})\quad\left(\Longleftrightarrow -A_{\circ}\ge \mathcal{F}_{-}(L_{-\circ})\right).
\end{gathered}
\end{equation}
Равенство в первом соотношении достигается на и только на контурах семейств $\Gamma_+(g_{\circ},A_{\circ})$ из теоремы \ref{theorem-optimal-curves},
а во втором -- достигается на и только на контурах семейств $\Gamma_-(g_{\circ},A_{\circ})$ из теоремы \ref{theorem-optimal-curves}, где элемент $g_{\circ}\in G$ произволен.
\end{theorem}
\begin{proof}
Утверждение непосредственно следует из теоремы \ref{theorem-optimal-curves} и предложения \ref{F-of-L-props}.
\end{proof}
Таким образом, в общем случае мы имеем два изопериметрических соотношения, совпадающих при $\Omega=-\Omega$, а обобщениями граничной гиперболы служат кривые $(L,\mathcal{F}_{+}(L)), L>0$ и $(L,-\mathcal{F}_{-}(L)), L>0$, которые могут быть заданы параметрически соотношениями \eqref{L-F-of-lambda-formulae-plus}, \eqref{L-F-of-lambda-formulae-minus}. Заметим, что в общем случае эти граничные кривые не являются гиперболами и даже могут не иметь асимптоту при $L\to+\infty$.

Для сравнения приведем изопериметрические соотношения на плоскости Минковского для евклидовой площади, полученные в работе \cite{Busemann1947TheIP}: $L^2_{\pm\circ}\ge 4S_{\Omega^{\circ}}A_{\circ}$. То есть, в этом случае аналогичные граничные кривые всегда являются параболами. Заметим также, что из предложения \ref{F-of-L-props} следует, что на финслеровой плоскости Лобачевского при малых $L$ верно $\mathcal{F_+}(L)= \frac{1}{4S_{\Omega^{\circ}}}L^2+o(L^2)$. Поэтому изопериметрическое неравенство на плоскости Минковского можно рассматривать как предельный случай изопериметрического соотношения на финслеровой плоскости Лобачевского при $L\to 0$.

\section{Примеры для p-кругов}
Рассмотрим множества $\Omega_{p}=\{(x,y):|x|^p+|y|^p\le 1\}, p\in [1,+\infty]$. Пусть $q=\frac{p}{p-1}$. Так как $\Omega_{p}=-\Omega_{p}$, то $L_-(\lambda)=L_+(-\lambda)$, $F_-(\lambda)=-F_+(-\lambda)$, и нам достаточно найти только функции $L_{+}(\lambda),F_+(\lambda),\lambda>1$. Непосредственные вычисления приводят к следующим результатам
$$
	L_{+}(\lambda)=
	\begin{cases} \frac{4\lambda}{\lambda^2-1}+2\ln(\frac{\lambda+1}{\lambda-1}), p=1,\\ 
		4\lambda\int\limits_{0}^{1}\frac{1}{(\lambda^2-x^2)(1-x^q)^{\frac{1}{p}}}dx, p\in (1,+\infty),\\
		2\ln(\frac{\lambda+1}{\lambda-1}), p=+\infty.
	\end{cases}
$$
$$
	F_+(\lambda)=
	\begin{cases}
		\frac{4}{\lambda^2-1},p=1,\\
		4\int\limits_{0}^{1}\frac{x^q}{(\lambda^2-x^2)(1-x^q)^{\frac{1}{p}}}dx,p\in (1,+\infty),\\
		2\ln(\frac{\lambda^2}{\lambda^2-1}),p=+\infty.
	\end{cases}
$$
$$
	a_{+}=
	\begin{cases}
		+\infty,p=1,\\
		4\int\limits_{0}^{1}\frac{(1-x^q)^{\frac{1}{q}}}{(1-x^2)}dx<+\infty,p\in (1,+\infty),\\
		4\ln 2,p=+\infty.
	\end{cases}
$$
Проверим, что результат теоремы \ref{theorem-iso-ineq} согласуется с результатом на классической плоскости Лобачевского.
\begin{proposition}
	\label{iso-ineq-circle}
	Для $\Omega=\Omega_{2}$ изопериметрические соотношения \eqref{iso-ineq} имеют хорошо известный вид:
	$$
	L_{+\circ}^2-4\pi A_{\circ}-A_{\circ}^2\ge 0.
	$$
\end{proposition}
\begin{proof}
	В данном случае функцию $F_+(\lambda)$ легко выразить через $L_+(\lambda)$, пользуясь тем, что $\Omega_2^{\circ}=\Omega_2$:
	$$
	F_+(\lambda)=\int_{0}^{2\pi}\frac{\cos(\theta)- \lambda + \lambda}{\lambda-\cos(\theta)}d\theta=\lambda L_{+}(\lambda)-2\pi.
	$$
	При $p=2$ интеграл в выражении $L_+(\lambda)$, конечно, считается явно. Или можно воспользоваться предложением \ref{L-A-DE}, что приводит к ОДУ
	$$
	L'_{+}=-\frac{\lambda}{(\lambda^2 - 1)} L_{+}.
	$$
	Учитывая пункт \ref{L-der-limit-infty} предложения \ref{L-props}, получаем:
	$$
	\begin{gathered}
		L_{+}(\lambda)=\frac{2\pi}{\sqrt{\lambda^2 - 1}},\quad F_+(\lambda)=2\pi(\frac{\lambda}{\sqrt{\lambda^2 - 1}} - 1).
	\end{gathered}
	$$
	Отсюда легко следует, что для любого $\lambda>1$ точка $\frac{1}{2\pi}(L_{+}(\lambda),F_+(\lambda)+2\pi)$ лежит на гиперболе $y^2-x^2=1$, а $\mathcal{L}_+(A)=\sqrt{A^2+4\pi A}$.
\end{proof}
Из явно выписанных формул для $L_+,~F_+$ получаем следующие соотношения.
\begin{remark}
	Для $\Omega_{1}$ изопериметрические соотношения \eqref{iso-ineq} имеют вид $$\frac{L_{+\circ}}{2}\ge \arcosh\left(\frac{A_{\circ}+2}{2}\right)+\sqrt{\left(\frac{A_{\circ}+2}{2}\right)^2-1}.$$
\end{remark}
\begin{remark}
	Для $\Omega_{\infty}$ изопериметрические соотношения \eqref{iso-ineq} имеют вид
	$$
	\cosh\left(\frac{L_{+\circ}}{4}\right)\ge e^{\frac{A_{\circ}}{4}}.
	$$
\end{remark}
\begin{remark}
В случае, когда $a_+<+\infty$, кажется разумным рассматривать <<нормированное>> изопериметрическое соотношение. А именно, обозначим $x_{\circ}=\frac{1}{a_+}L_{+\circ}$, $y_{\circ}=\frac{1}{a_+}(A_{\circ}+a_+)$. Тогда в случае $\Omega_2$ изопериметрическое соотношение может быть записано в виде $y_{\circ}^2-x_{\circ}^2\le 1$, а в случае $\Omega_{\infty}$ -- в виде $2^{x_{\circ}}+2^{-x_{\circ}}\ge 2^{y_{\circ}}$.
\end{remark}
\begin{figure}[h!]
	\center{\includegraphics[scale=0.8]{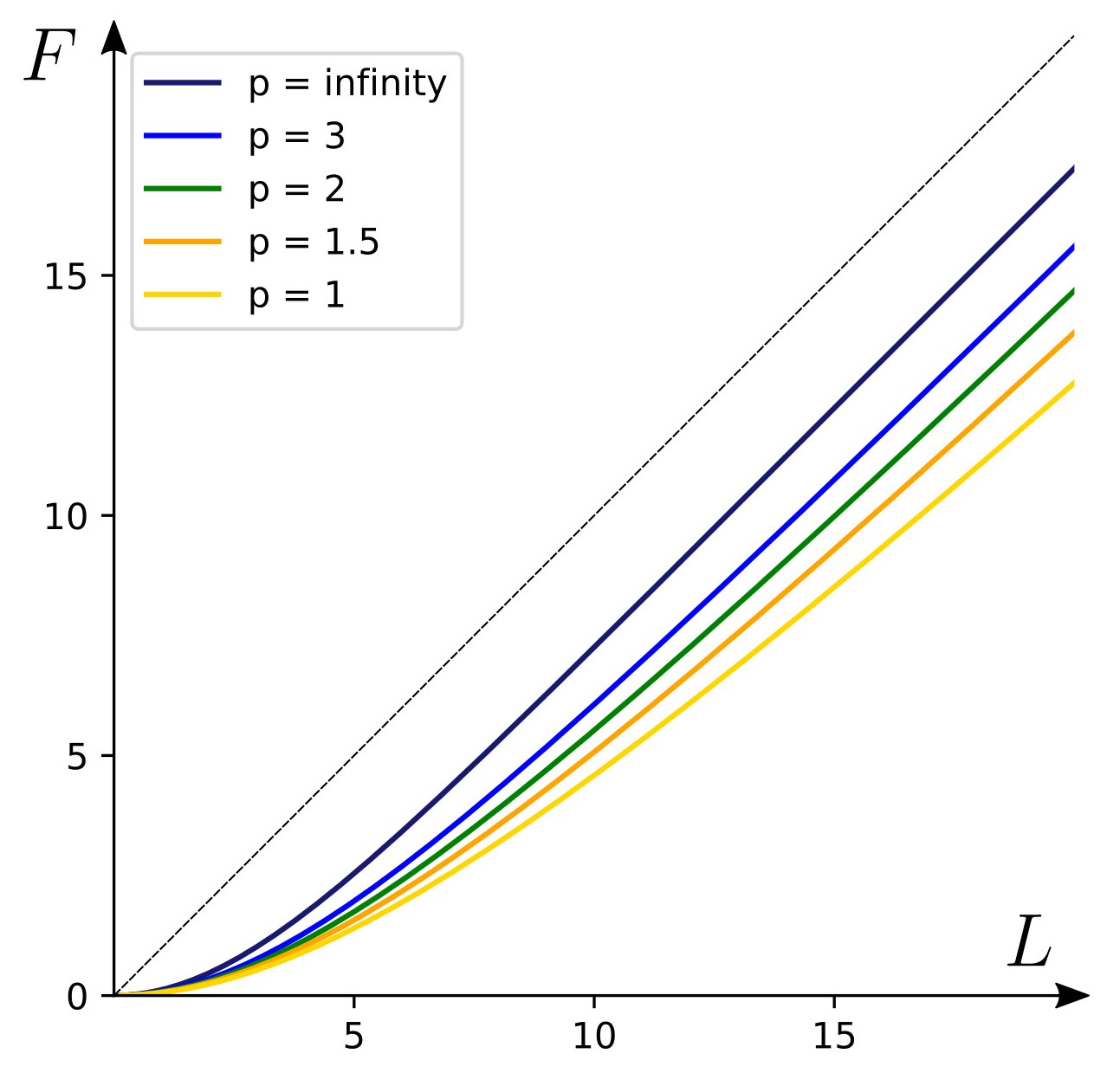}}
	\caption{Графики функций $F=\mathcal{F}_{+}(L)$ для множеств $\Omega_{p}=\{(x,y):|x|^p+|y|^p\le 1\}.$}
\end{figure}

\end{document}